\RequirePackage{silence}
\documentclass[11pt,letterpaper,reqno]{amsart}
\pdfoutput=1
\usepackage[ascii]{inputenc}
\usepackage[bookmarksnumbered,hypertexnames=false,colorlinks=true,linkcolor=blue,urlcolor=blue,citecolor=blue,anchorcolor=green,breaklinks=true,pdfusetitle]{hyperref}
\usepackage{bbm,microtype,mathrsfs,amsmath,amssymb,xcolor,amsthm,mathtools,graphicx,enumitem,caption,fullpage}
\usepackage[T1]{fontenc}
\usepackage{textcomp}
\usepackage[english]{babel}
\usepackage[capitalize]{cleveref}
\numberwithin{equation}{section}
\newtheorem{thm}{Theorem}[section]\crefname{thm}{Theorem}{Theorems}
\newtheorem{lem}[thm]{Lemma}\crefname{lem}{Lemma}{Lemmas}
\crefname{prb}{Problem}{Problems}
\crefname{rem}{Remark}{Remarks}
\newtheorem{cor}[thm]{Corollary}\crefname{cor}{Corollary}{Corollaries}
\newtheorem{conj}[thm]{Conjecture}\crefname{conj}{Conjecture}{Conjectures}
\newtheorem*{conj*}{Conjecture}
\newtheorem{dfn}[thm]{Definition}\crefname{dfn}{Definition}{Definitions}
\newtheorem{prp}[thm]{Proposition}\crefname{prp}{Proposition}{Propositions}
\crefname{obs}{Observation}{Observations}
\crefname{qn}{Question}{Questions}
\theoremstyle{definition}
\newtheorem{exa}[thm]{Example}\crefname{exa}{Example}{Examples}
\DeclareMathOperator{\tr}{tr}
\DeclareMathOperator{\ad}{ad}

\DeclareMathOperator{\Mat}{Mat}

\DeclareMathOperator{\End}{End}
\DeclareMathOperator{\diag}{diag}

\DeclareMathOperator{\GL}{GL}
\newcommand{\DKL}{\operatorname{D}_{\operatorname{KL}}}
\DeclareMathOperator{\Exp}{Exp}

\DeclareMathOperator{\Sym}{Sym}
\DeclareMathOperator{\poly}{poly}

\DeclareMathOperator{\prim}{pm}

\DeclareMathOperator{\cst}{Cst}
\DeclareMathOperator{\perm}{perm}
\DeclareMathOperator{\capa}{cap}
\DeclareMathOperator{\Capa}{Cap}
\DeclareMathOperator{\I}{I}

\DeclareMathOperator{\U}{U}
\DeclarePairedDelimiter{\abs}{\lvert}{\rvert}
\DeclarePairedDelimiter{\norm}{\lVert}{\rVert}
\DeclarePairedDelimiter{\braket}{\langle}{\rangle}
\newcommand{\CC}{\mathbb C}
\newcommand{\QQ}{\mathbb Q}
\newcommand{\EE}{\mathbb E}
\newcommand{\RR}{\mathbb R}
\newcommand{\PP}{\mathbb P}
\newcommand{\ZZ}{\mathbb Z}
\newcommand{\NN}{\mathbb N}
\newcommand{\ot}{\otimes}
\newcommand{\ga}{\mathfrak{g}}
\newcommand{\ka}{\mathfrak{k}}

\newcommand{\ta}{\mathfrak{t}}

\newcommand{\ma}{\mathfrak{m}}
\newcommand{\op}{\oplus}
\newcommand{\eps}{\varepsilon}
\newcommand{\id}{\mathbbm 1}

\newcommand*\diff{\mathop{}\!\mathrm{d}}
\newcommand{\du}{\diff{u}}
\newcommand{\duu}{\diff{\bar u}}
\newcommand{\dX}{\diff{X}}
\newcommand{\dY}{\diff{Y}}

\begin{document}
\title{Minimal length in an orbit closure as a semiclassical limit}
\author{Cole Franks}
\address{Cole Franks: Department of Mathematics, Massachusetts Institute of Technology.}
\email{franks@mit.edu}
\author{Michael Walter}
\address{Michael Walter:
Faculty of Computer Science, Ruhr University Bochum.
Korteweg-de Vries Institute for Mathematics, Institute for Theoretical Physics, Institute for Language, Logic, and Computation, and QuSoft, University of Amsterdam.}
\email{michael.walter@rub.de}
\begin{abstract}
Consider the action of a connected complex reductive group on a finite-dimensional vector space.
A fundamental result in invariant theory states that the orbit closure of a vector~$v$ is separated from the origin if and only if some homogeneous invariant polynomial is nonzero on~$v$, i.e. $v$ is not in the null cone. Thus, efficiently finding the minimum distance between the orbit closure and the origin can lead to deterministic algorithms for null cone membership, an important polynomial identity testing problem including the non-commutative Edmonds problem.
This connection to optimization has recently led to efficient algorithms for many problems in invariant theory.

Here we explore a refinement of the famous duality between orbit closures and invariant polynomials, which holds that the following two quantities coincide:
(1)~the logarithm of the Euclidean distance between the orbit closure and the origin and (2)~the rate of exponential growth of the \emph{invariant part} of~$v^{\ot k}$ in the semiclassical limit as $k$ tends to infinity. This result can be deduced from work of S.~Zhang (\emph{Geometric reductivity at Archimedean places}, 1994), which uses sophisticated tools in arithmetic geometry. We provide a new and independent elementary proof inspired by the Fourier-analytic proof of the local central limit theorem.
We generalize the result to projections onto highest weight vectors and isotypical components, and explore connections between such semiclassical limits and the asymptotic behavior of multiplicities in representation theory, large deviations theory in classical and quantum statistics, and the Jacobian conjecture as reformulated by Mathieu.
Our formulas imply that they can be computed, in many cases efficiently, to arbitrary precision.
\end{abstract}
\maketitle

\section{Introduction}\label{sec:intro}
Consider a vector~$v$ in a finite-dimensional representation space $V$ of a connected complex reductive group $G$.
Let $V$ carry an inner product invariant under a maximally compact subgroup~$K\subseteq G$.
A fundamental result in geometric invariant theory states that the orbit closure of~$v$ under~$G$ contains the zero vector ($v$ is \emph{unstable}) if and only if every nonconstant homogeneous $G$-invariant polynomial vanishes on~$v$ ($v$ is in the \emph{null cone})~\cite{hilbert1893vollen,mumford1994geometric}.
In this manuscript we prove quantitative sharpenings and extensions of this relationship.

We quantify the evaluations of invariant polynomials in degree~$k$ by the \emph{invariant part} of~$v^{\ot k}$, that is, its image under the orthogonal projection~$\Pi_k$ onto the invariant subspace of $V^{\ot k}$.
The sequence~$\norm{\Pi_k v^{\ot k}}$ is super-multiplicative, so it is natural to study the following semiclassical limit:%
\footnote{Considering the $\limsup$ rather than the limit is necessary only because the sequence may vanish outside an additive subsemigroup of $\NN$.
By Fekete's lemma, there is $a \in \NN$ such that restricted to the subsemigroup $\{k a: k \in \NN\}$ the sequence $\norm{\Pi_k v^{\ot k}}^{\frac1k}$ has a limit and outside of which $\norm{\Pi_k v^{\ot k}}^{\frac1k}$ eventually vanishes. See also \cref{lem:supmul}.}
\begin{gather}\label{eq:cap-proj}
  \limsup_{k \to \infty} \, \norm{\Pi_k v^{\ot k}}^{\frac1k}
\end{gather}
Because $v^{\ot k}$ is an element of the space of symmetric tensors, 
we may think of $\norm{\Pi_k v^{\ot k}}^2$ as a sum of squares $\sum_p\abs{p(v)}^2$ of evaluations of an orthonormal basis~$\{p\}$ of $G$-invariant homogeneous polynomials of degree~$k$.
Equivalently, it is equal to $\EE \abs{P(v)}^2$ where $P$~is a random homogeneous $G$-invariant polynomial of degree~$k$ drawn from the complex standard Gaussian distribution.
Thus, \cref{eq:cap-proj} measures the asymptotic behavior of \emph{random} invariant polynomials.

On the other hand, we may consider the Euclidean distance between the $G$-orbit closure of $v$ and the origin, sometimes called the \emph{capacity} of~$v$:
\begin{gather}\label{eq:cap-inf}
  \capa(v):=\inf_{g \in G} \norm{\phi(g) v},
\end{gather}
where $\phi\colon G\to\GL(V)$ denotes the representation.
By the relationship between instability and the null cone, \cref{eq:cap-proj} is zero if and only if \cref{eq:cap-inf} is zero.
Moreover, it is straightforward to prove that the semiclassical limit is a lower bound for the capacity.
The two quantities are, in fact, equal.

\begin{thm}\label{thm:main}
For all $v\in V$, we have
\begin{align*}
\limsup_{k \to \infty} \, \norm{\Pi_k v^{\ot k}}^{\frac1k} =  \inf_{g \in G} \, \norm{\phi(g) v}.
\end{align*}
\end{thm}

\noindent
The theorem can be deduced from \cite[Stmt.~1.6]{zhang1994geometric}, albeit in a non-obvious way.
Here we provide an independent and more elementary proof.
We now discuss several interpretations and extensions of the result.

\subsection{Invariant theory and optimization}\label{subsec:opt}
Optimization problems like \cref{eq:cap-inf} are known as \emph{scaling problems}.
When~$G$ is an Abelian group, scaling is equivalent to a class of optimization problems collectively known as \emph{unconstrained geometric programming}, which includes the \emph{matrix scaling} and \emph{matrix balancing} problems.
These problems have, respectively, important applications in optimal transport~\cite{cuturi2013sinkhorn} and preconditioning for numerical solvers~\cite{parlett1971balancing}.

In the general case, the problem of computing \cref{eq:cap-inf} can be interpreted as a non-commutative analogue of unconstrained geometric programming.
The non-commutative analogue of matrix scaling, called \emph{operator scaling}, has recently been leveraged to prove upper bounds in algebraic complexity~\cite{garg2016deterministic}, compute Brascamp-Lieb constants in analysis~\cite{garg2017algorithmic}, and obtain algorithmic guarantees for heavy tailed covariance estimation~\cite{franks2020rigorous}.
Scaling problems also capture versions of the \emph{quantum marginal problem}, in which one asks whether given mixed quantum states are the partial traces of some global pure state~\cite{klyachko2004quantum,daftuar2005quantum,christandl2006spectra,christandl2007nonzero,walter2013entanglement,burgisser2018alternating,burgisser2018efficient}, and in entanglement witnessing~\cite{walter2014multipartite}.
Recently it was shown how to compute the capacity in polynomial time in a fairly general setting which includes, for instance, the quiver representations with constantly many vertices~\cite{burgisser2019towards}.

We interpret the capacity as the \emph{dual} of the semiclassical limit in the sense of optimization; \cref{thm:main} shows that \emph{strong duality} holds.
In particular, \cref{thm:main} yields a formula for the semiclassical limit \cref{eq:cap-inf} that is computable to arbitrary precision, in many cases efficiently, without computing a single invariant polynomial. 

\subsection{Mathieu and Jacobian conjectures}\label{subsec:mathieu intro}
Let $\du$ denote the Haar probability measure on~$K$.
Then we may express the norm square of the invariant part by the following integral:
\begin{align}\label{eq:pi_k as integral}
  \lVert \Pi_k v^{\ot k}\rVert^2 = \int_K \braket{v, \phi(u) v}^k \, \du.
\end{align}
Thinking of $f(u) = \braket{v, \phi(u) v}$ as defining a complex-valued function on $K$, we see that $\lVert \Pi_k v^{\ot k}\rVert^2$ is the \emph{constant term} $\cst(f^k)$ of the Fourier expansion of~$f^k$ in the sense of the Peter-Weyl theorem.
Further, $f$ is~$K$-finite: it has finitely many terms in its Fourier expansion. The asymptotic behavior of $K$-finite functions is the subject of the following conjecture which, as shown by Mathieu~\cite{mathieu1995some}, implies Keller's Jacobian conjecture.

\begin{conj*}[Mathieu] Let $K$ be a compact connected Lie group, and let $f$ and $g$ be $K$-finite functions on $K$. If $\cst(f^k) = 0$ for all $k$, then $\cst(f^k g)=0$ for all but finitely many $k$.
\end{conj*}

\noindent
Duistermaat and van der Kallen~\cite{duistermaat1998constant} proved that Mathieu's conjecture holds if~$K$ is a torus.
In fact, they proved a stronger result that characterizes the sequence~$\cst(f^k)$ in terms of the critical values of $f$ on the \emph{complexification} $G$ of $K$.
For simplicity we state their theorem when~$K=\U(1)$ is the circle, in which case the $K$-finite functions are the Laurent polynomials, and $G = \CC_\times$.
Suppose~$f \in \CC[z, z^{-1}]$ is neither a polynomial in $z$ nor in $z^{-1}$.
Then, $f$ has a critical value~$\nu$ on~$\CC_\times$, and
\begin{gather}\label{eq:vdk}
  \limsup_{k \to \infty} \, \abs{\cst(f^k)}^{\frac1k} = \abs\nu > 0.
\end{gather}
The class of functions of the form $f(u) = \braket{v, \phi(u) v}$ are precisely the \emph{positive-definite} $K$-finite functions, those $K$-finite functions whose Fourier series have positive semidefinite components.
In \cref{subsec:mathieu}, we observe that the Mathieu conjecture holds when $f$ is positive-definite by considering the moment polytope of the orbit closure of the vector~$v$.

In light of this, one might conjecture that, analogously to \cref{eq:vdk}, the quantity $\limsup_{k \to \infty} \cst(f^k)$ may be characterized in terms of the critical points of $f$ on the complexification $G$ of $K$ for all $K$-finite $f$.
In the positive-definite case, our \cref{thm:main} does precisely this.
Finding a common generalization of our formula and \cref{eq:vdk} for all $K$-finite functions, especially one implying the Mathieu conjecture, remains an outstanding open problem.

\subsection{Moment polytopes}\label{subsec:moment intro}
So far we have only discussed invariants, but projections to other isotypic components and their highest weight spaces tell a similar story to \cref{thm:main}.

We begin with the highest weight spaces, as it is the key tool in our treatment of the isotypic components.
One is interested in the orthogonal projections~$\Pi^+_{k,\lambda}$ of~$V^{\ot k}$ to the subspace of highest weight vectors of weight~$\lambda$, where $\lambda$ is a dominant weight.
Recall that the dominant weights are the integral points in the \emph{positive Weyl chamber}.
It is well-known that the set of~$\lambda/k$ such that the component $\Pi^+_\lambda v^{\ot k}$ is nonzero is the set of rational points in a convex polytope with rational vertices, which we here call the \emph{Borel moment polytope}~\cite{guillemin2006convexity}.
This polytope is the moment polytope of the orbit closure of the ray through~$v$ in~$\PP(V)$ under the opposite Borel subgroup.

We sharpen \cite{guillemin2006convexity} by defining a $\log$-concave function $\theta \mapsto \capa_\theta(v)$ with support the Borel moment polytope, and then showing that~$\capa_\theta(v)$ matches the semiclassical limit of~$\norm{\Pi^+_{k,k\theta}(v^{\ot k})}^{\frac1k}$.
To be more specific, for $\theta$ in the positive Weyl chamber of $G$, we define
\begin{align*}
  \capa_\theta(v) :=\inf_{g\in G} \, \abs{\chi_{-\theta}(g)} \, \norm{\phi(g) v},
\end{align*}
where the additional term $\abs{\chi_{-\theta}(g)}$ is a factor we call the \emph{absolute character} (to be defined rigorously in \cref{sec:generalization}).
For now, we just remark that if $\lambda$ is a dominant weight, hence the highest weight of an irreducible representation~$\phi_\lambda\colon G\to\GL(V_\lambda)$ with highest weight vector~$v_\lambda$, then $\abs{\chi_{-\lambda}(g)} = \norm{\phi_\lambda^*(g) v_\lambda^\dagger}$. 
Observe that $\capa_0(v) = \capa(v)$.
As promised, we have the following theorem generalizing \cref{thm:main}.

\begin{thm}\label{thm:cap theta main}
For all $v\in V$ and for all rational~$\theta$ in the positive Weyl chamber, we have
\begin{align*}
  \limsup_{k\to\infty} \, \norm{\Pi^+_{k,k\theta} v^{\ot k}}^{\frac1k} = \capa_\theta(v), 
\end{align*}
where we take $\Pi^+_{k,\lambda} := 0$ if $\lambda$ is not a dominant weight.
\end{thm}

We now turn to the projections to isotypic components, which are to the \emph{moment polytope} of the $G$-orbit closure as $\Pi^+$ is to the Borel moment polytope~\cite{mumford1994geometric,brion1987image}.
Define $\Pi_{k, \lambda}$ as the orthogonal projection to the isotypic component of type $\lambda$ in $V^{\ot k}$.
In place of $\capa_\theta$ we consider its supremum over $K$-orbits,
\begin{align*}
  \Capa_\theta(v) := \sup_{u \in K} \capa_\theta(\phi(u) v).
\end{align*}

\begin{thm}\label{thm:Cap theta main}
For all $v\in V$ and for all rational~$\theta$ in the positive Weyl chamber, we have
\begin{align*}
  \limsup_{k\to\infty} \, \norm{\Pi_{k,k\theta} v^{\ot k}}^{\frac1k} = \Capa_\theta(v),
\end{align*}
where we take $\Pi_{k,\lambda} := 0$ if $\lambda$ is not a dominant weight.
\end{thm}

\noindent
We now discuss how \cref{thm:cap theta main,thm:Cap theta main} relate to some familiar measures on the moment polytope.

\subsection{Measures and multiplicities}\label{sec:large-dev}
The study of the asymptotic distribution of multiplicities in the semiclassical limit has a long history, see for instance the classical work by Heckman~\cite{heckman1982projections,guillemin1982geometric}.

As~$V^{\ot k}$ decomposes completely into isotypical components, the numbers $\norm{\Pi_\lambda v^{\ot k}}^2$ similarly form a probability distribution over the types $\lambda$ if $v$ is a unit vector.
Consider the random variable~$Y_k = \lambda/k$, which takes values in the positive Weyl chamber.
For commutative~$G$, the law of large numbers implies that~$Y_k$ converges in probability to~$\mu(v)$, the image of $v$ under the \emph{moment map} $\mu: V \to (i\ka)^*$, where~$\ka$ denotes the Lie algebra of~$K$. 
In general, the natural random variable~$X_k$ is defined on $(i\ka)^*$ rather than the positive Weyl chamber, and is equal to $\ad^*(u)Y_k$, where $u \in K$ is sampled with density proportional to $\norm{\Pi_{k, k Y_k}^+ (\phi(u^{-1}) v)^{\ot k}}^2$ with respect to the Haar measure.
Then we have the following result which generalizes the law of large numbers to noncommutative groups.

\begin{cor}\label{cor:intro convergence}
Let $v \in V$ be a unit vector.
Then, the random variables $X_k$ converge in probability to the moment map image~$\mu(v)$, and the random variables $Y_k$ converges in probability to the intersection of the coadjoint orbit~$\mathcal O_{\mu(v)}$ with the positive Weyl chamber.
\end{cor}

We note that \cref{cor:intro convergence} readily implies the classical result that, in the semiclassical limit~$k\to\infty$, the distribution of the dimensions of the isotypical components in~$\Sym^k(V)$ converges weakly to the \emph{Duistermaat-Heckman measure}, i.e., the distribution of the intersection of $\mathcal O_{\mu(v)}$ with the positive Weyl chamber for random~$v$ chosen from the Haar measure on the unit sphere~\cite{heckman1982projections} (cf.~\cite{guillemin1982geometric,sjamaar1995holomorphic,meinrenken1996riemann,meinrenken1999singular,vergne1998quantization,okounkov1996brunn,christandl2014eigenvalue} for related results and generalizations).

To prove \cref{cor:intro convergence}, we note that \cref{thm:cap theta main,thm:Cap theta main} in fact imply that $X_k$, $Y_k$ obey \emph{large deviations principles} with rate functions~$\theta \mapsto -\!\log \capa^2_\theta(v)$ and $\theta \mapsto -\!\log \Capa^2_\theta(v)$, respectively (see \cref{subsec:large dev}).
Moreover, it is known that $\log \capa_\theta(v) = 0$ if and only if $\mu(v) = 0$;
in fact for every unit vector~$v$ and $\theta$ in the moment polytope we have
\begin{gather*}
  \capa^2_\theta(v) \leq 1 - c \norm{\mu(v) - \theta}^2
\end{gather*}
for some constant $c>0$ that depends only on the representation~\cite{burgisser2019towards}.
Combining this with \cref{thm:cap theta main,thm:Cap theta main} yields \cref{cor:intro convergence}.

For commutative groups, the above large deviations statement is equivalent to Sanov's theorem.
In the noncommutative setting, a large deviations result had previously been established by Keyl and Werner in the context of \emph{quantum state estimation} and \emph{spectrum estimation}~\cite{keyl2006quantum,keyl2005estimating}.
Here, $K = U(n)$ acts on square complex matrices $A \in \Mat_{n}(\CC)$ by left multiplication, in which case $\mu(A) = A A^\dagger/\norm{A}_F^2$.
Then the random variables~$X_k$ obey a large deviations principle with a very explicit rate function, which agrees with the quantum relative entropy when~$\mu(A)$ is diagonal.
As a corollary, the random variables~$Y_k$, which serve as an estimator for the \emph{spectrum} of~$\mu(A)$, satisfy a large deviations principle with rate equal to the relative entropy.
A closely related example is the large deviations theorem for the multiplicity distributions of tensor power representations due to Duffield~\cite{duffield1990large} (cf.~\cite{postnova2020multiplicities}).
Our results generalize these examples to arbitrary representations.
We discuss this in more detail in \cref{subsec:large dev}.

\subsection{Related work}
Closely related work has been done independently by Alonso Botero, Matthias Christandl, and P\'eter Vrana~\cite{christandl2020large}.
Both this work and~\cite{christandl2020large} were originally done without awareness of~\cite{zhang1994geometric}; we are thankful to Geordie Williamson and Jean-Beno\^it Bost for pointing out this reference to us after the fact.
The deduction of \cref{thm:main} from~\cite[Stmt.~1.6]{zhang1994geometric} requires some reformulations, so the relationship is not obvious at first sight.

\subsection{Organization of the paper}
In \cref{sec:basic}, we formally defined the capacity and establish basic properties of the capacity and the semiclassical limit.
We provide a proof for \cref{thm:main} in \cref{sec:proof}.
In \cref{sec:generalization}, we define the capacity for arbitrary points in the positive Weyl chamber and generalize our main result from invariants to highest weight vectors and irreducible representations (\cref{thm:cap theta main,thm:Cap theta main}).
In \cref{sec:applications}, we discuss applications and examples of our results to measures on the moment polytope, large deviations, and the Mathieu conjecture.

\section{Capacity and basic properties}\label{sec:basic}
In this section, we define the capacity and establish some basic properties.
Let~$K$ be a compact connected Lie group, $G$ its complexification (a connected complex reductive algebraic group),
and $\phi \colon G\to\GL(V)$ a finite-dimensional rational representation on a complex vector space~$V$.
We write $G \cdot v = \{ \phi(g) v : g \in G\}$ for the $G$-orbit of a vector $v\in V$, and denote the $G$-invariant subspace by $V^G = \{ v \in V : \phi(g) v = v \; \forall g \in G \}$.
Finally, we choose a norm~$\norm{\cdot}$ on~$V$ that is induced by a $K$-invariant Hermitian inner product~$\braket{\cdot,\cdot}$ (by convention, linear in the second argument).

\begin{dfn}[Capacity]\label{dfn:cap}
Given a vector $v\in V$, we define its \emph{capacity} by
\begin{align*}
  \capa(v) = \inf_{g\in G} \, \norm{\phi(g) v} = \min_{w \in \overline{G \cdot v}} \, \norm w.
\end{align*}
\end{dfn}

\noindent
This definition from~\cite{burgisser2019towards} generalizes the notions of matrix and operator capacity developed in~\cite{gurvits1998deflation,gurvits2004classical} as well as the polynomial capacity of~\cite{gurvits2006hyperbolic}.

The capacity quantifies the basic notions of stability from geometric invariant theory~\cite{mumford1994geometric}.
To see this, recall that a vector~$v$ is called \emph{unstable} if $\overline{G \cdot v} \ni 0$.
Recall that the Euclidean and Zariski closure of $G$-orbits coincide.
Equivalently, $v$ is unstable if $p(v)=p(0)$ for every $G$-invariant polynomial~$p$ on~$V$.
The set of all unstable vectors is called the \emph{null cone} of~$V$.
If $v$ is not unstable, it is called \emph{semistable}.
Clearly, $\capa(v)>0$ if and only if $v$ is semistable.

It is also interesting to discuss when $\capa(v) = \norm{v}$, which means that $v$ has minimal norm in its $G$-orbit.
This can also be characterized infinitesimally.
For this, let $\ga$ and $\ka$ denote the Lie algebras of $G$ and $K$, respectively, so that $\ga = \ka \op i\ka$ as real vector spaces.
Moreover, 
$K\backslash G \cong \exp(i\ka)$.
Let $\Phi\colon\ga\to \End(V)$ denote the infinitesimal action of the Lie algebra.
That is, $\exp(\Phi(X)) = \phi(\exp(X))$. 
Since the norm is $K$-invariant, we can consider the \emph{Kempf-Ness function} $K\backslash G \to \RR$, $Kg \mapsto \log \,\norm{\phi(g)v}$ for fixed $0\neq v\in V$~\cite{kempf1979length}.
Its differential at the identity coset depends only on the ray $[v] \in \PP(V)$ and can be computed as follows~\cite{ness1984stratification}:
\begin{align}\label{eq:moment map}
  \mu\colon \PP(V) \to (i\ka)^*, \quad [v] \mapsto \left( X \mapsto \frac{\braket{v, \Phi(X) v}}{\norm{v}^2} \right).
\end{align}
The $K$-equivariant map $\mu$ thus defined is called the \emph{moment map}.
It is a moment map for the $K$-action on $\PP(V)$ in the sense of symplectic geometry~\cite{ness1984stratification,kirwan1984cohomology}.
We will often write $\mu(v) = \mu([v])$ if $v$ is a unit vector.

Importantly, $\mu([v])=0$ if, and only if, $v$ has minimal norm in its $G$-orbit.
This result is part of the Kempf-Ness theorem, and it follows from the geodesic convexity of the Kempf-Ness function~\cite{kempf1979length}.
As a consequence, $\capa(v) = \norm v$ if, and only if, $\mu([v])=0$.
Recent work has made this statement quantitative~\cite{burgisser2019towards}, and we will come back to this in \cref{subsec:measures}.
It is known that the minimum in $\capa(v) = \min_{w\in\overline{G\cdot v}} \, \norm w$ is attained on a $K$-orbit in the \emph{unique} closed orbit in $\overline{G\cdot v}$.
This implies that the capacity is constant on $G$-orbit closures.


Now let~$\Pi_k$ denotes the orthogonal projection onto the $G$-invariant subspace of $V^{\ot k}$.
The latter is equipped with the canonical inner product and norm induced from~$V$.
It is easy to see that capacity is an upper bound to~$\norm{\Pi_k v^{\ot k}}^{\frac1k}$ for any~$k$.
We will latter show that, asymptotically, the two quantities coincide.

\begin{lem}\label{lem:weak duality}
Let~$v\in V$.
Then, $\capa(v) \geq \norm{\Pi_k v^{\ot k}}^{\frac1k}$ for all~$k\in\NN$.
\end{lem}
\begin{proof}
  For every $g\in G$ and $k\in\NN$,
  \begin{align*}
    \norm{\phi(g) v}
  = \norm{\phi(g)^{\ot k} v^{\ot k}}^{\frac1k}
  \geq \norm{\Pi_k \phi(g)^{\ot k} v^{\ot k}}^{\frac1k}
  = \norm{\Pi_k v^{\ot k}}^{\frac1k}.
  \end{align*}
  The inequality is valid since the orthogonal projection $\Pi_k$ cannot increase the norm.
  The last step holds since $\Pi_k$ projects onto the invariant subspace.
  Now the claim follows by taking the infimum over $g\in G$. 
\end{proof}


\begin{lem}\label{lem:supmul}
Let $v\in V$.
Then, the sequence $\norm{\Pi_k v^{\ot k}}$ is super-multiplicative in $k\in\NN$, i.e.,
\begin{align*}
  \norm{\Pi_{k+l}(v^{\ot (k+l)})} \geq \norm{\Pi_k v^{\ot k}} \norm{\Pi_l (v^{\ot l})}
\end{align*}
for all $k, l\in\NN$.
As a consequence,
\begin{align}\label{eq:dual as sup}
  \limsup_{k\to\infty} \, \norm{\Pi_k v^{\ot k}}^{\frac1k} = \sup_{k\in\NN} \, \norm{\Pi_k v^{\ot k}}^{\frac1k}.
\end{align}
\end{lem}
\begin{proof}
Clearly, $(V^{\ot k})^G \otimes (V^{\ot l})^G \subseteq (V^{\ot (k+l)})^G$, so we have the operator inequality $\Pi_{k+l} \geq \Pi_k \ot \Pi_l$.
Thus we obtain
\begin{align*}
  \norm{\Pi_{k+l}(v^{\ot (k+l)})}^2
&= \braket{v^{\ot (k+l)}, \Pi_{k+l} (v^{\ot (k+l)})} \\
&\geq \braket{v^{\ot (k+l)}, (\Pi_k \ot \Pi_l) (v^{\ot (k+l)})}
= \norm{\Pi_k v^{\ot k}}^2 \norm{\Pi_l(v^{\ot l})}^2,
\end{align*}
which proves supermultiplicativity.

We now show \cref{eq:dual as sup}.
Define $p_k := \norm{\Pi_k v^{\ot k}}$ for $k\in \NN$.
It suffices to prove that
\begin{align*}
  \limsup_{k\to\infty} p_k^{\frac1k} \geq \sup_{k\in\NN} p_k^{\frac1k},
\end{align*}
since the other inequality is trivial.
We claim that in fact the following holds for any fixed $k\in\NN$:
\begin{align*}
  \lim_{n\to\infty} p_{nk}^{\frac1{nk}} \geq p_k^{\frac1k}
\end{align*}
To see this, we may assume that $p_k>0$, since otherwise the inequality holds trivially.
By the first part of the lemma, $(p_k)_{k\in\NN}$ is a super-multiplicative sequence.
Thus, if $p_k>0$ then $p_{nk}>0$ for all~$n\in\NN$, so it follows from Fekete's lemma applied to the super-additive sequence $\log p_k$ that
\begin{align*}
  \lim_{n\to\infty} p_{nk}^{\frac1{nk}} = \sup_{n\in\NN} p_{nk}^{\frac1{nk}} \geq p_k^{\frac1k}.
\end{align*}
\end{proof}

The quantity $\norm{\Pi_k v^{\ot k}}$ can also be interpreted in the language of invariant polynomials.
Indeed, note that~$v^{\ot k}$ is an element of the space of symmetric tensors, which is a subrepresentation of $V^{\ot k}$ and dual to the space~$\CC[V]_k$ of homogeneous polynomials on~$V$ of degree~$k$.
Thus,
\begin{align}\label{eq:proj vs polys}
  \norm{\Pi_k v^{\ot k}}^2 = \sum_i \abs{p_{k,i}(v)}^2,
\end{align}
where $\{p_{k,i}\}_i$ is an orthonormal basis of $\CC[V]_k$ with respect to the induced (Bombieri) inner product.
Equivalently, we may write \cref{eq:proj vs polys} as $\norm{\Pi_k v^{\ot k}}^2 = \EE\abs{P_k(v)}^2$, where $P_k$ is a complex standard Gaussian random homogeneous $G$-invariant polynomial of degree~$k$.
Thus:
\begin{align*}
  \limsup_{k\to\infty} \, \norm{\Pi_k v^{\ot k}}^{\frac1k}
= \limsup_{k\to\infty} \, \bigl( \EE\abs{P_k(v)}^2 \bigr)^{1/2k}
= \limsup_{k\to\infty} \, \bigl( \EE\abs{P_k(v)} \bigr)^{1/k},
\end{align*}
where the last inequality follows because $\dim \Sym^k V$ is polynomial in $k$.

\section{Proof of strong duality}\label{sec:proof}
In this section we prove \cref{thm:main}, which states that the limsup of $\norm{\Pi_k v^{\ot k}}^{\frac1k}$ is given by the capacity.
The following proposition contains our main computation.
The key idea is to use \cref{eq:pi_k as integral} to write
\begin{align}\label{eq:invariants via haar}
  \norm{\Pi_k v^{\ot k}}^2
= \braket{v^{\ot k}, \Pi_k v^{\ot k}}
= \int_K \braket{v^{\ot k}, \phi(u)^{\ot k} v^{\ot k}} \, \du
= \int_K \braket{v, \phi(u) v}^k \, \du,
\end{align}
where $\du$ denotes the Haar probability measure on~$K$, and to evaluate the dominant contribution to this integral for large~$k$.
Given a vector~$v\in V$, we write~$K_v$ for the $K$-stabilizer of~$v$ and, if $v\neq0$, we write $K_{[v]}$ for the $K$-stabilizer of~$[v]\in\PP(V)$.
Moreover, recall that $\mu$ denotes the moment map defined in \cref{eq:moment map}.

\begin{prp}\label{prp:main}
  Let $v\in V$ be a unit vector with $\mu(v)=0$.
  Then there exists an integer $m>0$ such that
  \begin{align}\label{eq:lim claim}
    \lim_{k \to \infty} (mk)^{(\dim K - \dim K_{[v]})/2} \, \norm{\Pi_{mk} v^{\ot mk}}^2
  > 0,
  \end{align}
  i.e., the limit exists and is positive.
\end{prp}
\begin{proof}
Since $\mu(v)=0$, the vector~$v$ is semistable.
This in turn implies that $K_{[v]}/K_v$ is necessarily finite~\cite[Lemma~2.2]{ness1984stratification}.
Thus, $K_{[v]}$ acts on the one-dimensional subspace~$\CC v$ by a finite subgroup of~$\U(1)$.
It follows that there is a positive integer $m>0$ such that $(\phi(u) v)^{\ot m} = v^{\ot m}$ for every~$u\in K_{[v]}$.
Using \cref{eq:invariants via haar}, we find that
\begin{align}\label{eq:integral quotient}
  \norm{\Pi_{mk} v^{\ot mk}}^2
= \int_K \braket{v, \phi(u) v}^{mk} \, \du
= \int_{K/K_{[v]}} \braket{v, \phi(u) v}^{mk} \, \duu,
\end{align}
where we denote by $\duu$ the unique left-$K$-invariant probability measure on $K/K_{[v]}$ (see, e.g.,~\cite[Theorem 8.36]{knapp2013lie}).
Note that $\abs{\braket{v,\phi(u)v}}\leq1$ for every $u\in K$, since $K$ acts by unitaries.
Moreover, equality holds if and only if $u\in K_{[v]}$, by equality condition for the Cauchy-Schwarz inequality.
The advantage of the right-hand side expression in \cref{eq:integral quotient} is that here this bound is only saturated at a single point, namely at the identity coset in $K/K_{[v]}$.

Since $K$ is compact, there exists an inner product on~$\ka$ that is invariant under the adjoint action of~$K$.
This inner product gives rise to a bi-invariant Riemannian metric and volume form on~$K$, which induces the Haar measure~$\du$ (provided the inner product is suitably normalized).
Moreover, $K/K_{[v]}$ is a normal homogeneous space.
Namely, if $\ka_{[v]}$ denote the Lie algebra of $K_{[v]}$ then we can identify its orthogonal complement~$\ma := \ka_{[v]}^\perp \subseteq \ka$ with the tangent space of $K/K_{[v]}$ at the identity coset.
The restriction of the inner product to $\ma$ then induces a left-invariant Riemannian metric and volume form on $K/K_{[v]}$, which induces the quotient measure~$\duu$.
Moreover, the projection $K\to K/K_{[v]}$ is a Riemannian submersion, and the Riemannian exponential map at the identity coset is given by $\Exp\colon \ma \to K/K_{[v]}$, $X \mapsto \overline{\exp(X)}$, where $\exp$ denotes the exponential map from $\ka$ to $K$~\cite[Theorem 3.65]{gallot1990riemannian}.

We can use this to localize the integral to a small neighborhood of the identity coset.
Let $\eps>0$ be small enough such that~$\Exp$ is a local diffeomorphism from an $\eps$-ball around the origin, denoted~$B_\eps(0)$, onto an open neighborhood of the identity coset, denoted~$U$.
We can then split the integral in \cref{eq:integral quotient} into two terms, one over the neighborhood~$U$ and one over its complement:
\begin{align}\label{eq:split}
  \int_{K/K_{[v]}} \braket{v, \phi(u) v}^{mk} \, \duu
= \int_U \braket{v, \phi(u) v}^{mk} \, \duu
+ \int_{U^c} \braket{v, \phi(u) v}^{mk} \, \duu.
\end{align}
It is easy to see that the second term in \cref{eq:split} does not contribute to \cref{eq:lim claim}.
Indeed, since~$U^c$ is compact and $\abs{\braket{v, \phi(u) v}}<1$ for $u\not\in K_{[v]}$, there exists a constant~$C<1$ such that $\abs{\braket{v, \phi(u) v}} \leq C$ for all~$u\in K^c$.
Then, $\abs{\int_{U^c} \braket{v, \phi(u) v}^{mk} \, \duu} \leq C^{mk}$, and it follows that
\begin{align}\label{eq:second term vanishes}
  \lim_{k\to\infty} (mk)^{(\dim K - \dim K_{[v]})/2} \int_{U^c} \braket{v, \phi(u) v}^{mk} \, \duu \to 0.
\end{align}

We now consider the first term in \cref{eq:split}.
We can write
\begin{align*}
  \int_U \braket{v, \phi(u) v}^{mk} \, \duu
= \int_{B_\eps(0)} \braket{v, \exp(\Phi(X)) v}^{mk} \, J(X) \, \dX,
\end{align*}
where $\dX$ denotes the Lebesgue measure induced by the inner product on~$\ma$, and $J(X)$ denotes the Jacobian of~$\Exp$, the Riemannian exponential map, at~$X\in\ma$.
At $X=0$, the differential of the~$\Exp$ is an isometry, so $J(0)=1$.
Moreover, $J(X)$ is a smooth function of~$X$.
By choosing~$\eps$ to be sufficiently small, we may therefore assume that $J(X)\leq2$ for all~$X\in B_\eps(0)$.
Next, make the change of variables $Y=\sqrt{mk} X$, so that
\begin{align}\label{eq:subs}
  (mk)^{(\dim K - \dim K_{[v]})/2} &\int_{B_\eps(0)} \braket{v, \exp(\Phi(X)) v}^{mk} \, J(X) \, \dX = \int_{\ma} h_k(Y) \, \dY,
\end{align}
where
\begin{align}\label{eq:subs integrand}
  h_k(Y)
= \braket{v, \exp\Bigl(\Phi\bigl(\tfrac Y{\sqrt{mk}}\bigr)\Bigr) v}^{mk} \, \id_{B_\eps(0)}\bigl(\tfrac Y{\sqrt{mk}}\bigr) \, J\bigl(\tfrac Y{\sqrt{mk}}\bigr).
\end{align}
Here, $\id_{B_\eps(0)}$ is the indicator function of the $\eps$-ball, and we used that $\dim \ma = \dim K - \dim K_{[v]}$.

We will compute the limit of \cref{eq:subs} as $k\to\infty$ by using the dominated convergence theorem.
To start, define $f_X(t) := \braket{v, \exp(\Phi(Xt)) v}$ for $X\in\ma$.
Then:
\begin{align*}
  f_X(0) &= \norm{v}^2 = 1, \\
  f_X'(0) &= \braket{v, \Phi(X) v} = \mu(v)(X) = 0, \\
  f_X''(0) &= \braket{v, \Phi(X)^2 v} = -\norm{\Phi(X) v}^2 = O(\norm X^2), \\
  \abs{f_X'''(t)} &= \abs{\braket{v, \Phi(X)^3 \exp(\Phi(Xt)) v}} \leq \norm{\Phi(X)^3} \leq \norm{\Phi(X)}^3 = O(\norm X^3)
\end{align*}
Here, we used that $v$ is a unit vector, that $\mu(v)=0$, and that $K$ acts unitarily; the constants hidden in the big O notation only depend on the operator norm of the Lie algebra representation.
By choosing~$\eps$ to be sufficiently small, we may assume that
$\abs{\braket{v, \exp(\Phi(X)) v} - 1} \leq 1/2$ for all~$X\in B_\eps(0)$.
Then it follows from the Taylor expansion of $\log f_X$, where $\log$ denotes the principal branch of the logarithm, that
\begin{align}\label{eq:uniform estimate}
  \braket{v, \exp(\Phi(X)) v}
= f_X(1)
= e^{\log f_X(1)}
= e^{-\frac12\norm{\Phi(X) v}^2 + O(\norm{X}^3)}
\end{align}
for all~$X\in B_\eps(0)$.

We now determine the pointwise limit of the integrand \cref{eq:subs integrand}.
For any fixed~$Y$ and sufficiently large~$k$, we can apply \cref{eq:uniform estimate} with $X=\frac Y{\sqrt{mk}}$, so we  obtain
\begin{align*}
  \lim_{k\to\infty} \braket{v, \exp\Bigl(\Phi\bigl(\tfrac Y{\sqrt{mk}}\bigr)\Bigr) v}^{mk}
= \lim_{k\to\infty} e^{-\frac12\norm{\Phi(Y) v}^2 + \frac{O(\norm{Y}^3)}{\sqrt k}}
= e^{-\frac12\norm{\Phi(Y) v}^2}.
\end{align*}
Since both $\id_{B_\eps(0)}\bigl(\tfrac Y{\sqrt k}\bigr)$ and $J\bigl(\tfrac Y{\sqrt k}\bigr)$ converge to~$1$ as $k\to\infty$, it follows that
\begin{align}\label{eq:pointwise limit}
  \lim_{k\to\infty} h_k(Y) = h(Y) := e^{-\frac12\norm{\Phi(Y) v}^2}.
\end{align}
is the pointwise limit of the functions $h_k(Y)$.
Note that $Q(Y) := \norm{\Phi(Y) v}^2$ is a positive definite quadratic form on $\ma$.
This is because $Q(Y)=0$ implies that $Y\in\ka_{[v]} = \ma^\perp$.
Thus, the integral of \cref{eq:pointwise limit} is a finite Gaussian integral, namely
\begin{align}\label{eq:pointwise integral}
  \int_{\ma} h(Y) \, \dY
= \int_{\ma} e^{-\frac12 Q(Y)} \, \dY
= \sqrt{\frac{(2\pi)^{\dim \ma}}{\det(Q)}}
> 0,
\end{align}
where $\det(Q)$ denotes the determinant of the positive definite matrix that corresponds to the quadratic form~$Q$.

We can similarly show that the integrand \cref{eq:subs integrand} is upper-bounded in absolute value by an integrable function that is independent of $k$.
Indeed, since $Q$ is positive definite, it follows from \cref{eq:uniform estimate} that we can choose~$\eps$ small enough such that
\begin{align}\label{eq:bound in ball}
  \abs{\braket{v, \exp(\Phi(X)) v}}
\leq e^{-\frac14 Q(X)}
\end{align}
for all $X\in B_\eps(0)$.
Then,
\begin{align*}
  \abs{h_k(Y)}
= \abs[\Big]{\braket{v, \exp\Bigl(\Phi\bigl(\tfrac Y{\sqrt{mk}}\bigr)\Bigr) v}^{mk} \, \id_{B_\eps(0)}\bigl(\tfrac Y{\sqrt{mk}}\bigr) \, J\bigl(\tfrac Y{\sqrt{mk}}\bigr)}
\leq 2 e^{-\frac{mk}4 Q\bigl(\tfrac Y{\sqrt{mk}}\bigr)}
= 2 e^{-\frac14 Q(Y)},
\end{align*}
where we used \cref{eq:bound in ball} and that $\abs{J(X)}\leq2$ for $X\in B_\eps(0)$.
The right-hand side function is integrable, again because $Q$ is positive definite.

Thus, the dominated convergence theorem is applicable and shows that
\begin{align*}
  \lim_{k\to\infty} \int_{\ma} h_k(Y) \, \dY
= \int_{\ma} h(Y) \, \dY
= \sqrt{\frac{(2\pi)^{\dim \ma}}{\det(Q)}}
> 0
\end{align*}
using \cref{eq:pointwise limit,eq:pointwise integral}.
In view of \cref{eq:subs,eq:second term vanishes,eq:split,eq:integral quotient}, we have proved the proposition.
\end{proof}

We now prove \cref{thm:main}.

\begin{proof}[Proof of \cref{thm:main}]
It remains to show that
\begin{align*}
  \limsup_{k\to\infty} \, \norm{\Pi_k v^{\ot k}}^{\frac1k} \geq \capa(v),
\end{align*}
since we already showed the easy converse direction (\cref{lem:weak duality}).
For this, let $w$ be a vector of minimal norm in~$\overline{G \cdot v}$, so that $\capa(v) = \norm{w}$, while
\begin{align*}
  \limsup_{k\to\infty} \, \norm{\Pi_k v^{\ot k}}^{\frac1k} = \limsup_{k\to\infty} \, \norm{\Pi_k w^{\ot k}}^{\frac1k},
\end{align*}
since $v\mapsto \norm{\Pi_k v^{\ot k}}^{\frac1k}$ is constant on $G$-orbit closures and the same is true for the limsup.
We may assume that~$w\neq 0$, since otherwise there is nothing to prove.
Then, $\mu([w])=0$, since $w$ is in particular a nonzero vector of minimal norm in its own $G$-orbit.
Thus it suffices to prove
\begin{align*}
  \limsup_{k\to\infty} \, \norm{\Pi_k w^{\ot k}}^{\frac1k} \geq \norm w.
\end{align*}
Since both the left-hand side and the right-hand side are homogeneous of degree one, we may further assume that~$\norm{w}=1$.
Then it suffices to exhibit a subsequence of $k$'s such that $\norm{\Pi_k w^{\ot k}} = \Omega(1/\!\poly(k))$.
This is achieved by \cref{prp:main}.
\end{proof}

\section{Generalization to highest weights and isotypical components}\label{sec:generalization}
There are two natural generalizations of the notion of an invariant vector: highest weight vectors and irreducible representations.
Accordingly, in this section we will consider \emph{two} generalizations of the capacity and show that they correspond precisely to the asymptotics of these two representation-theoretic notions, respectively (\cref{thm:cap theta main,thm:Cap theta main}).

Let $K \subseteq T$ be a maximal torus, $B \subseteq G$ a Borel subgroup containing $T$, and $N \subseteq B$ its maximal unipotent subgroup.
Let $\ta$ denote the Lie algebra of~$T$.
Let $\Lambda \subseteq (i\ta)^*$ the weight lattice of~$T$ 
and $\Lambda_+ \subseteq \Lambda$ the semigroup of dominant weights with respect to $B$.
The positive Weyl chamber $C_+$ is the convex polyhedral cone spanned by $\Lambda_+$.
Let $\ad^*$ denote the coadjoint action of $K$ on~$(i\ka)^*$.
It is well-known that $C_+$ is a cross-section for this action, that is, each coadjoint orbit intersects $C_+$ in a single point.
Here we consider $C_+ \subseteq (i\ka)^*$ by using an inner product that is invariant under the adjoint action of~$K$ on $i\ka$.
For each~$\lambda\in\Lambda_+$, denote by~$\phi_\lambda\colon G\to\GL(V_\lambda)$ the irreducible representation of~$G$ with highest weight~$\lambda$.
We equip~$V_\lambda$ with a $K$-invariant inner product, also denoted~$\braket{\cdot,\cdot}$, and fix a highest weight vector~$v_\lambda\in V_\lambda$ of unit norm (which is unique up to phase).

Finally, let $A = \exp(i\ta)$, $B^- \subseteq B$ the Borel subgroup opposite to~$B$ and $N^- \subseteq B^-$ its maximal unipotent subgroup.
Then we have an Iwasawa decomposition $G=KAN^-$, where $A = \exp(i\ta)$ and we use the opposite unipotent subgroup $N^-$ for reasons that will become clear shortly.
Denote by $a\colon G\to A$ the function that assigns to a group element its component in~$A$ acccording to the Iwasawa decomposition.

\begin{dfn}[Absolute character]\label{dfn:abs-char}
For $\theta\in C_+$, define the \emph{absolute character} as the function
\begin{align*}
  \abs{\chi_{-\theta}}\colon G\to \RR_+, \quad g \mapsto \abs{\chi_{-\theta}(g)} := e^{-\theta(\log a(g))}.
\end{align*}
\end{dfn}

\noindent
In general, we have the following equivariance property:
For all $g\in G$ and $b\in B^-$,
\begin{align}\label{eq:abs char cov}
  \abs{\chi_{-\theta}(gb)} = \abs{\chi_{-\theta}(g)} \, \abs{\chi_{-\theta}(b)},
\end{align}
so in particular $\abs{\chi_{-\theta}}$ is a character of~$B^-$.

The absolute character has the following representation theoretic interpretation.
Let $\lambda\in\Lambda_+$ be a dominant weight and $v_\lambda\in V_\lambda$ a unit-norm highest weight vector, then
\begin{align}\label{eq:abs char vs lwv}
  \abs{\chi_{-\lambda}(g)} = \norm{\phi_\lambda^*(g) v_\lambda^\dagger}
\end{align}
for all $g\in G$, where $\phi_\lambda^*$ denotes the dual representation.
Indeed, $v_\lambda^\dagger = \braket{v_\lambda,\cdot} \in V_\lambda^*$ is a lowest weight vector in the dual representation, of weight $-\lambda$. Moreover, we have the formula
\begin{align}\label{eq:unitary equivariance}
  \abs{\chi_{-\theta}(g)} = \abs{\chi_{-\theta}(gu)}
\end{align}
for all $\theta\in C_+$, $g\in G$, and $u\in K_\theta$.
This follows because both $v_\lambda$ and $v_\lambda^\dagger$ are invariant under $K_\lambda$, the stabilizer of $\lambda$ under the coadjoint action of~$K$.
This invariance extends directly to rational~$\theta=\frac\lambda\ell$ and, by continuity, to arbitrary~$\theta$ in the positive Weyl chamber.

\bigskip

We now give our first generalization of the capacity.

\begin{dfn}[$\theta$-capacity]
Given $\theta\in C_+$, we define the \emph{$\theta$-capacity} of~$v\in V$ by
\begin{align*}
  \capa_\theta(v)
= \inf_{g\in G} \, \abs{\chi_{-\theta}(g)} \, \norm{\phi(g) v}
= \inf_{H\in i\ta, n \in N^-} \, e^{-\theta(H)} \, \norm{e^{\Phi(H)} \phi(n) v}.
\end{align*}
We also denote the supremum of the $\theta$-capacity over $K$-orbits by
\begin{align*}
  \Capa_\theta(v)
= \sup_{x\in\mathcal O_\theta} \capa_x(v)
= \sup_{u\in K} \capa_\theta(\phi(u)v).
\end{align*}
\end{dfn}

\noindent This definition is very similar to the one in~\cite{burgisser2019towards}, except that here we choose to work with lowest weight vectors rather than highest weight vectors of the dual representation.

The $\theta$-capacity is semi-invariant under the opposite Borel subgroup: For all $b\in B^-$,
$\capa_\theta(\phi(b^{-1}) v) = \abs{\chi_{-\theta}(b)} \capa_\theta(v)$,
as follows from \cref{eq:abs char cov}. 
Moreover, $\capa_\theta(v)$ is log-concave as a function of~$\theta\in C_+$.

For $\theta=0$, we recover the definition of the capacity from \cref{dfn:cap}.
In fact, $\capa_\theta(v)$ is for any rational~$\theta$ just an ordinary capacity in disguise:
For all $\lambda\in\Lambda_+$ and $\ell\in\NN$, \cref{eq:abs char vs lwv} implies that
\begin{align}\label{eq:cap theta as cap}
  \capa_{\frac\lambda\ell}(v)
= \left( \inf_{g\in G} \, \abs{\chi_{-\lambda}(g)} \, \norm{\phi(g) v}^\ell \right)^{\frac1\ell}
= \capa(v^{\ot\ell} \ot v_\lambda^\dagger)^{\frac1\ell}
\end{align}
where the right-hand side capacity is computed in the representation $V^{\ot \ell} \ot V_\lambda^*$.
Thus, \cref{thm:main} immediately yields an interpretation in terms of the asymptotic growth of invariants.

We now prove \cref{thm:cap theta main}, which states that the $\theta$-capacity equals the limsup of $\norm{\Pi^+_{k,k\theta} v^{\ot k}}^{1/k}$, i.e., measures the growth of the projection of $v^{\ot k}$ onto the subspace of highest weight vectors.

\begin{proof}[Proof of \cref{thm:cap theta main}]
Let $\ell\in\NN$ be the smallest number such that $\lambda := \ell\theta \in \Lambda_+$.
From \cref{eq:cap theta as cap,thm:main}, we find that
\begin{align}
  \capa_\theta(v)
= \capa(v^{\ot\ell} \ot v_\lambda^\dagger)^{\frac1\ell}
= \limsup_{k\to\infty} \, \norm{\Pi'_k \bigl((v^{\ot\ell} \ot v_\lambda^\dagger)^{\ot k}\bigr)}^{\frac1{k\ell}},\label{eq:shift}
\end{align}
with $\Pi'_k$ the orthogonal projection onto the $G$-invariant subspace of~$W^{\ot k}$, where $W = V^{\ot\ell} \ot V_\lambda^*$.
Since $v_\lambda^{\ot k}$ is a highest weight vector of weight~$k\lambda$, it is the highest weight vector of a copy of $V_{k \lambda}$ inside $V_\lambda^{\ot k}$. As $\Pi_k'$ commutes with orthogonal projections to subrepresentations, the norm in the right-hand side of \cref{eq:shift} can equivalently be written as
\begin{align*}
  \norm{\Pi'_k \bigl((v^{\ot\ell} \ot v_\lambda^\dagger)^{\ot k}\bigr)}
= \norm{\Pi''_k (v^{\ot k\ell} \ot v_{k\lambda}^\dagger)},
\end{align*}
with $\Pi''_k$ the orthogonal projection onto the $G$-invariant subspace of $V^{\ot k\ell} \ot V_{k\lambda}^*$.
We claim that
\begin{align}\label{eq:claim}
  \norm{\Pi''_k (v^{\ot k\ell} \ot v_{k\lambda}^\dagger)} = \frac1{\sqrt{d_{k\lambda}}} \norm{\Pi^+_{k\ell,k\lambda} (v^{\ot k\ell})},
\end{align}
where $d_\nu := \dim V_\nu$.
This will complete the proof, because $d_{k\lambda}$ grows only polynomially with~$k$ by the Weyl dimension formula, hence
\begin{align*}
  \capa_\theta(v)
= \limsup_{k\to\infty} \left( \frac1{\sqrt{d_{k\lambda}}} \norm{\Pi^+_{k\ell,k\lambda} (v^{\ot k\ell})} \right)^{\frac1{k\ell}}
= \limsup_{k\to\infty} \, \norm{\Pi^+_{k\ell,k\lambda} (v^{\ot k\ell})}^{\frac1{k\ell}}.
\end{align*}
We now proceed with the proof of \cref{eq:claim}. By complete reducibility, it suffices to show that, for any two highest weights $\alpha,\beta\in\Lambda_+$, $w\in V_\alpha$, and $\Pi$ the projection onto the $G$-invariant subspace of $V_\alpha \ot V_\beta^*$, we have
\begin{align}\label{eq:schur}
  \norm{\Pi(w \ot v_\beta^\dagger)} = \begin{cases}
    0 & \text{ if } \alpha\neq\beta, \\
    \frac1{\sqrt{d_\beta}} \abs{\braket{v_\beta,w}} & \text{ if } \alpha=\beta,
  \end{cases}
\end{align}
where $v_\beta$ denotes a unit-norm highest weight vector in $V_\beta$.
For this, recall that $\dim (V_\alpha \ot V_\beta^*)^G = \delta_{\alpha,\beta}$ by Schur's lemma.
If $\alpha\neq\beta$, this means that $\Pi=0$, so the first statement is clear.
If $\alpha=\beta$, then the one-dimensional $G$-invariant subspace is spanned by the normalized identity operator~$I/\sqrt{d_\beta}$, which is a unit vector in $V_\beta \ot V_\beta^* \cong L(V_\beta)$ (the induced inner product is the Hilbert-Schmidt inner product).
Thus,
\begin{align*}
  \norm{\Pi(w \ot v_\beta^\dagger)} = \frac1{\sqrt{d_\beta}} \abs{\tr(w v_\beta^\dagger)} = \frac1{\sqrt{d_\beta}} \abs{\braket{v_\beta,w}},
\end{align*}
which establishes the second statement in \cref{eq:schur}, and thereby \cref{eq:claim}.
\end{proof}

Next, we generalize \cref{lem:supmul}.

\begin{lem}\label{lem:supmul plus}
Let $v\in V$, $k,l\in\NN$, and $\lambda,\nu\in\Lambda_+$. Then,
\begin{align}\label{eq:generalized supmul}
  \norm{\Pi_{k+l,\lambda+\nu}(v^{\ot (k+l)})} \geq \norm{\Pi^+_{k,\lambda} (v^{\ot k})} \norm{\Pi^+_{l,\nu} (v^{\ot l})}.
\end{align}
As a consequence, the sequence $\norm{\Pi^+_{k,k\theta} (v^{\ot k})}$ is super-multiplicative in $k\in\NN$ for any $\theta\in\QQ_+\Lambda_+$, and
\begin{align}\label{eq:dual theta as sup}
  \limsup_{k\to\infty} \, \norm{\Pi^+_{k,k\theta} (v^{\ot k})}^{\frac1k} = \sup_{k\in\NN} \, \norm{\Pi^+_{k,k\theta} (v^{\ot k})}^{\frac1k}.
\end{align}
\end{lem}
\begin{proof}
The tensor product of two highest weight vectors with weight~$\lambda$ and~$\nu$, respectively, is a highest weight vector of weight~$\lambda+\nu$.
This shows the operator inequality $\Pi^+_{k+l,\lambda+\nu} \geq \Pi^+_{k,\lambda} \ot \Pi^+_{l,\nu}$.
Now \cref{eq:generalized supmul} and the remaining statements follow just like in the proof of \cref{lem:supmul}.
\end{proof}

We now prove \cref{thm:Cap theta main}, which shows that an easy modification of the $\theta$-capacity computes the limsup of $\norm{\Pi_{k,k\theta} v^{\ot k}}^{1/k}$, the asymptotic growth of the projection of $v^{\ot k}$ onto the isotypical component (rather than the subspace of highest weight vectors).
Namely, we only need to replace $\capa_\theta(v)$ by $\Capa_\theta(v)$, its supremum over the $K$-orbit of~$v$.

\begin{proof}[Proof of \cref{thm:Cap theta main}]
In view of \cref{thm:cap theta main}, we need to show that
\begin{align*}
  \sup_{u\in K} \limsup_{k\to\infty} \, \norm{\Pi^+_{k,k\theta}((\phi(u)v)^{\ot k})}^{\frac1k}
= \limsup_{k\to\infty} \, \norm{\Pi_{k,k\theta}(v^{\ot k})}^{\frac1k}.
\end{align*}
Again, one inequality is easy.
Since $\Pi^+_{k,\lambda} \leq \Pi_{k,\lambda}$ for every $k\in\NN$ and $\lambda\in\Lambda_+$, we have
\begin{align*}
\sup_{u\in K} \limsup_{k\to\infty} \, \norm{\Pi^+_{k,k\theta}((\phi(u)v)^{\ot k})}^{\frac1k}
\leq \sup_{u\in K} \limsup_{k\to\infty} \, \norm{\Pi_{k,k\theta}((\phi(u)v)^{\ot k})}^{\frac1k}
= \limsup_{k\to\infty} \, \norm{\Pi_{k,k\theta}(v^{\ot k})}^{\frac1k},
\end{align*}
using $K$-invariance in the last step.

Next, we must show the reverse inequality.
For this, note that we can write
\begin{align}\label{eq:int-proj}
  \Pi_{k,\lambda} = d_{\lambda} \int_K \phi(u^{-1})^{\ot k} \Pi^+_{k,\lambda} \phi(u)^{\ot k} \, \du
\end{align}
where $d_\lambda := \dim V_{\lambda}$ if $\lambda\in V_\lambda$, and $d_\lambda:=0$ otherwise.
Thus,
\begin{align*}
  \norm{\Pi_{k,k\theta}(v^{\ot k})}
\leq d_{k\theta} \int_K \norm{\phi(u^{-1})^{\ot k} \Pi^+_{k,k\theta} (\phi(u) v)^{\ot k}} \, \du
= d_{k\theta} \int_K \norm{\Pi^+_{k,k\theta} (\phi(u) v)^{\ot k}} \, \du.
\end{align*}
It follows that for every $k\in\NN$ there exists $u_k\in K$ such that
\begin{align*}
  \norm{\Pi_{k,k\theta}(v^{\ot k})} \leq d_{k\theta} \norm{\Pi^+_{k,k\theta} (\phi(u_k) v)^{\ot k}}.
\end{align*}
Since $d_{k\theta}$ grows only polynomially with~$k$ by the Weyl dimension formula, we find that
\begin{align*}
\limsup_{k\to\infty} \, \norm{\Pi_{k,k\theta}(v^{\ot k})}^{\frac1k}
&\leq \limsup_{k\to\infty} d_{k\theta}^{\frac1k} \, \norm{\Pi^+_{k,k\theta} (\phi(u_k) v)^{\ot k}}^{\frac1k}
= \limsup_{k\to\infty} \, \norm{\Pi^+_{k,k\theta} (\phi(u_k) v)^{\ot k}}^{\frac1k} \\
&\leq \sup_{u\in K} \sup_{k\in\NN} \, \norm{\Pi^+_{k,k\theta} (\phi(u) v)^{\ot k}}^{\frac1k}
= \sup_{u\in K} \limsup_{k\to\infty} \, \norm{\Pi^+_{k,k\theta} (\phi(u) v)^{\ot k}}^{\frac1k},
\end{align*}
where the last step is due to \cref{eq:dual theta as sup}.
This concludes the proof.
\end{proof}

\section{Applications and examples}\label{sec:applications}

In this section we describe more carefully some of the connections discussed in \cref{sec:intro}.

\subsection{Capacity and moment map}\label{subsec:moment map}
Here we expand upon \cref{subsec:moment intro}; for more background see~\cite{kempf1979length,ness1984stratification,kirwan1984cohomology,guillemin2006convexity,burgisser2019towards}.
Recall from \cref{sec:basic} that $\capa(v) = \norm v$ iff $\mu([v])=0$.
Moreover, $\capa(v) > 0$ if and only if there exists $w\in\overline{G \cdot v}$ such that~$\mu([w])=0$, or if and only if there exists $G$-invariant polynomial $p$ such that $p(v)\neq p(0)$.

We can similarly characterize other points in the image of the moment map.
For this, it is useful to extend the definition of the $\theta$-capacity to arbitrary points~$x=\ad^*(u)\theta$ in~$(i\ka)^*$ by
\begin{align}\label{eq:def cap x}
  \capa_x(v) := \capa_\theta(\phi(u^{-1})v).
\end{align}
This is well-defined by \cref{eq:unitary equivariance}.
Now one can similarly show that $\capa_x(v) = \norm{v}$ if and only if~$\mu(v) = \theta$.
Moreover, $\capa_\theta(v) > 0$ if and only if there exists $w \in \overline{B_- \cdot [v]}$ such that~$\mu([w]) = \theta$.

The image under the moment map of the $G$-orbit closure of~$[v]$ in~$\PP(V)$ is by $K$-equivariance closed under the coadjoint action, so fully characterized by the \emph{moment polytope}:
\begin{align*}
  \Delta(v) := \mu\bigl(\overline{G \cdot [v]}\bigr) \cap C_+
  = \bigl\{ \theta \in C_+ \;\big\vert\; \mu([w]) \in \mathcal O_\theta, \, [w] \in \overline{G \cdot [v]} \bigr\},
\end{align*}
where $\mathcal O_\theta = \ad^*(K)\theta$ denotes the coadjoint orbit through~$\theta$.
By Mumford's theorem, $\Delta(v)$ is a convex polytope with rational vertices~\cite{ness1984stratification,kirwan1984cohomology,brion1987image}.
From the preceding discussion it is clear 
that $\Capa_\theta(v) = \norm v$ if and only if $\mu([v]) \in \mathcal O_\theta$, and $\Capa_\theta(v) > 0$ if and only if $\theta\in\Delta(v)$.
Thus, the support of the function $\theta \mapsto \Capa_\theta(v)$ is precisely the moment polytope.
Now we see that \cref{thm:Cap theta main} implies the well-known result that the rational points of the moment polytope are determined by the highest weights that occur in the homogeneous coordinate ring of the orbit closure.
That is, for~$\theta\in\QQ_+\Lambda_+$, we have that~$\theta \in \Delta(v)$ if and only if~$V^*_{k\theta}$ occurs in~$\CC[\overline{G \cdot [v]}]_k$, the degree-$k$ part of the homogeneous coordinate ring of the orbit closure, for some $k\in\NN$.

\subsection{Measures and multiplicities}\label{subsec:measures}
It is natural to study the growth of multiplicities in the homogeneous coordinate ring of a projective $G$-variety.
Heckman studied this question for the restriction of coadjoint orbits~\cite{heckman1982projections}; see~\cite{guillemin1982geometric,sjamaar1995holomorphic,meinrenken1996riemann,meinrenken1999singular,vergne1998quantization,okounkov1996brunn} for related results and generalizations.
For $\PP(V)$, his result can be stated as follows.
Let $d=\dim V$, $d_\lambda=\dim V_\lambda$, and denote by $m_{k,\lambda}$ denotes the multiplicity of~$V_\lambda$ in~$\Sym^k(V)$.
Then the sequence of probability measures
\begin{align}\label{eq:nu heckman}
  \mu_{V,k} = \frac1{\binom{d+k-1}k} \sum_{\lambda\in\Lambda_+} d_\lambda m_{k,\lambda} \, \delta_{\frac\lambda k}
\end{align}
converges weakly to the so-called Duistermaat-Heckman measure, defined as the pushforward of a Haar random vector in~$\PP(V)$ along the moment map and further onto the positive Weyl chamber.
This result is often stated in a slightly different but equivalent way, leaving out the dimensions~$d_\lambda$.

We can refine Heckman's result by replacing the multiplicities by projecting tensor powers of a fixed vector onto subspaces of highest weight vectors.
By \cref{eq:int-proj},
\begin{align}\label{eq:povm}
  \sum_{\lambda\in\Lambda_+} d_{\lambda} \int_K \phi(u)^{\ot k} \Pi^+_{k,\lambda} \phi(u^{-1})^{\ot k} \, \du
= \sum_{\lambda\in\Lambda_+} \Pi_{k,\lambda} = I,
\end{align}
so we can for any~$k\in\NN$ and unit vector~$v\in V$ define a probability measure on $K \times \Lambda_+$ by
\begin{align}\label{eq:prod measure}
  \diff\nu^{(v)}_k(u,\lambda) = d_{\lambda} \, \norm{\Pi^+_{k,\lambda} (\phi(u^{-1}) v)^{\ot k}}^2 \, \du\diff\lambda
\end{align}
where $\du$ denotes the Haar measure on~$K$ and $\diff\lambda$ the counting measure on $\Lambda_+$.
By \cref{eq:unitary equivariance}, the density at $(u,\lambda)$ only depends on the point $\ad^*(u)\lambda \in \mathcal O_\lambda$, which motivates the following definition.

\begin{dfn}\label{dfn:X_k}
For any unit vector $v\in V$ and $k\in\NN$, define the random variable~$X_k(v) = \ad^*(u) \frac\lambda k$ in $(i\ka)^*$, where $(u,\lambda)$ is drawn from the probability measure defined in \cref{eq:prod measure}.
\end{dfn}

\noindent
That is, for any measurable function $f$ on $(i\mathfrak k)^*$ we have that
\begin{align}\label{eq:X as int}
  E\bigl[f(X_k(v))\bigr]
= \sum_{\lambda\in\Lambda_+} d_\lambda \int_K \norm{\Pi^+_{k,\lambda} (\phi(u^{-1}) v)^{\ot k}}^2 \, f\bigl(\ad^*(u) \tfrac{\lambda}k\bigr) \, \du.
\end{align}

\begin{thm}\label{thm:measure}
For any unit vector $v\in V$, $X_k(v)$ converges in probability to the constant $\mu(v)$.
\end{thm}
\begin{proof}
To prove this, we will show that
\begin{align*}
   \Pr\bigl[\norm{X_k(v) - \mu(v)} \geq \eps\bigr] \to 0
\end{align*}
for any $\eps>0$, where $\norm{\cdot}$ denotes the norm on $(i\ka)^*$ discussed earlier.
For this, consider \cref{eq:X as int} with~$f$ the indicator function of the set $\{ x \in (i\mathfrak k)^* : \norm{x - \mu(v)} \geq \eps \}$.
Since the dimension of $\Sym^k(V)$ grows polynomially with~$k$, there are only polynomially many~$\lambda$ to consider.
Moreover, $d_\lambda$ grows only polynomially with~$k$ by the Weyl dimension formula.
Thus it suffices to show that
\begin{align*}
  \norm{\Pi^+_{k,\lambda} (\phi(u^{-1}) v)^{\ot k}}^2
\end{align*}
decays exponentially with~$k$ provided $\norm{\mu(v) - \ad^*(u)\theta} \geq \eps$.
By \cref{thm:cap theta main,eq:dual theta as sup},
\begin{align*}
  \sup_{k\in\NN} \, \norm{\Pi^+_{k,k\theta}((\phi(u^{-1}) v)^{\ot k})}^{\frac1k}
= \capa_\theta(\phi(u^{-1})v)
\end{align*}
so we only need to upper bound the right-hand side capacity by a number strictly smaller than 1 that works uniformly for all $(\theta,u)$ such that $\norm{\mu(v) - \ad^*(u)\theta} \geq \eps$.

For this, we use the result from~\cite{burgisser2019towards} that there exists a constant $c>0$, depending only on the representation, such that
\begin{align*}
  \capa^2_\theta(w) \leq 1 - c \norm{\mu([w]) - \theta}^2
\end{align*}
for all $w\in V$ and $\theta\in\QQ_+\Lambda_+$ such that $\capa_\theta(w)>0$.
If we apply this to~$w=\phi(u^{-1})v$, we obtain
\begin{align*}
  \capa^2_\theta(\phi(u^{-1})v)
\leq 1 - c \norm{\ad^*(u^{-1})\mu(v) - \theta}^2
= 1 - c \norm{\mu(v) - \ad^*(u)\theta}^2,
\end{align*}
since the moment map is $K$-equivariant and the norm invariant under the coadjoint action.
This concludes the proof.
\end{proof}

As a corollary, we obtain the limit of the random variables~$Y_k(v)$ that take value~$\frac\lambda k \in C_+$ with probability~$\norm{\Pi_{k,\lambda} v^{\ot k}}^2$.
Let $s\colon (i\ka)^*\to C_+$ denote the map that sends $\mathcal O_\theta \mapsto \theta$ for any $\theta\in C_+$.

\begin{cor}\label{cor:measure}
For any unit vector $v\in V$, $Y_k(v)$ converges in probability to~$s(\mu(v))$.
\end{cor}
\begin{proof}
By \cref{eq:povm}, $Y_k(v)$ has the same distribution as $s(X_k(v))$, so the result follows at once.
\end{proof}

Note that \cref{thm:measure,cor:measure} jointly establish \cref{cor:intro convergence} in the introduction.
The preceding results strengthen Heckman's theorem, which we now recover as a corollary.

\begin{cor}[Heckman]
The measures $\mu_{V,k}$ defined in \cref{eq:nu heckman} converge weakly to the Duistermaat-Heckman measure, i.e., the distribution of $s(\mu(v))$ for a Haar random unit vector $v\in V$.
\end{cor}
\begin{proof}
We know from \cref{cor:measure} that for each $v\in v$, $Y_k(v)$ converges weakly to $s(\mu(v))$.
Now let $\diff v$ denote the Haar measure on the unit sphere of~$V$.
Then the measure $\int \PP^{Y_k(v)} \, \diff v$ converges weakly to the distribution of $s(\mu(v))$ for Haar random~$v$.
On the other hand,
\begin{align*}
  \int \PP^{Y_k(v)} \, \diff v
= \int \sum_{\lambda \in \Lambda_+} \norm{\Pi_{k,\lambda} v^{\ot k}}^2 \delta_{\frac\lambda k}\, \diff v
= \int \sum_{\lambda \in \Lambda_+} \tr(\Pi_{k,\lambda} v^{\ot k} v^{\ot k,\dagger}) \, \delta_{\frac\lambda k} \diff v.
\end{align*}
This equals $\mu_{V,k}$, since $\binom{d+k-1}k \int v^{\ot k} v^{\ot k,\dagger} \, \diff v$ is the projection onto the symmetric subspace of~$V^{\ot k}$.
\end{proof}

\subsection{Large deviations}\label{subsec:large dev}
We now discuss a different motivation coming from the theory of large deviations.
In the previous section, we saw that the random variables~$X_k(v)$ and~$Y_k(v)$ converge in measure to the moment map~$\mu(v)$ and the intersection~$s(\mu(v))$ of its coadjoint orbit with the positive Weyl chamber.
The proofs hinged on viewing \cref{thm:cap theta main} as expressions for the rate of exponential decay of the densities of these random variables.

In fact, \cref{thm:cap theta main} implies that~$\{X_k(v)\}_{k\in\NN}$ satisfies a \emph{large deviations principle} with \emph{rate function} $x \mapsto -\!\log\capa^2_x(v)$, with $\capa_x(v)$ defined as in \cref{eq:def cap x}.
That is,
\begin{align*}
  \sup_{x\in S^\circ} \log\capa^2_x(v)
\leq \liminf_{k \to \infty}\frac{1}{k} \log \Pr[X_k(v) \in S]
\leq \limsup_{k \to \infty}\frac{1}{k} \log \Pr[X_k(v) \in S]
\leq \sup_{x\in \bar S} \log\capa^2_x(v)
\end{align*}
for all Borel measurable sets $S\subseteq (i\ka)^*$.
We note that the above optimizations can be restricted to the moment map image of~$\overline{G\cdot[v]}$, since otherwise $\capa_x(v)=0$ as discussed in \cref{subsec:moment map}.
As a consequence, $\{Y_k(v)\}_{k\in\NN}$ also satisfies a large deviations principle, with rate function~$\theta \mapsto -\!\log\Capa^2_\theta(v)$. 
This can also be seen directly using \cref{thm:Cap theta main}.

In what follows, by computing $\capa_\theta$ in a few special cases, we reproduce several results from the theory of large deviations, namely Sanov's theorem, Keyl and Werner's results on quantum tomography, and Duffield's large deviations principle for multiplicities.

\subsubsection*{Sanov's theorem}
We start with Sanov's theorem, a classical result in large deviations theory~\cite{sanov1957probability}.
Let~$\mathcal P_n$ denote the simplex of probability mass funtions over the finite alphabet~$\{1,\dots,n\}$.

\begin{thm}[Sanov]\label{thm:sanov}
Let~$X_k$ denote the empirical distribution of~$k$ independent samples from a distribution~$q\in\mathcal P_n$.
Then, $X_k$ obeys a large deviations principle with rate function~$\DKL(\cdot\Vert q)$, where $\DKL(p\Vert q) := \sum_{k=1}^n p_k \log(p_k/q_k)$ denotes the Kullback--Leibler divergence or relative entropy.
\end{thm}

\noindent
Before proceeding with the proof, let us discuss how the setting of Sanov's theorem is a special case of our setting.
Let $G$ denote the complex torus~$\CC_\times^n$.
Since the group is abelian, $K=T$, and we may identify $(i\ka)^* = \RR^n$ and $\Lambda = \ZZ^n$.
Let $G$ act on $V=\CC^n$ by coordinate-wise multiplication, i.e. $(\phi(g)v)_k = g_k v_k$ for $k\in[n]$.
Finally, define $v\in V$ such that $\abs{v_k}^2 = q_k$ for $k\in [n]$.
We now make two observations:
\begin{enumerate}
\item The $X_k$ are distributed as the random variables $X_k(v)$ defined above in \cref{dfn:X_k}.
\item The moment map image $\mu(v)$ is equal to $q$.
\end{enumerate}
Thus \cref{thm:sanov} follows from the following computation which holds for all $\theta\in\mathcal P_n$:
\begin{align*}
  -\log\capa^2_\theta(v)
= \sup_{x\in\RR^n} \Bigl( \theta\!\cdot\!x - \log \sum_{k=1}^n e^{x_k} q_k \Bigr)
= \DKL(\theta\Vert q)
\end{align*}

In fact, for abelian groups~$G=T$ we have a much more general formula.
Let $\mathcal P(\Omega)$ denote the set of probability mass functions over~$\Omega \subseteq \Lambda$, the finite set of weights of the representations.
Then we have for~$\theta\in\mathcal P(\Omega)$ and any unit vector $v\in V$ that
\begin{align}\label{eq:cap-relent}
\begin{array}{ccrl}
  -\log\capa_\theta^2(v) &=& \min &\DKL(p\Vert q) \\
  &&\text{subject to} & \sum_{\omega\in\Omega} p_\omega \, \omega = \theta, \quad p \in \mathcal P(\Omega)
\end{array}
\end{align}
where $q_\omega = \norm{P_\omega v}^2$.
This formula is straightforward to prove using Lagrange multipliers, and is itself a consequence of Sanov's theorem.
In turn, Sanov's theorem can be recovered from the general formula by specializing to the representation of $G=\CC_\times^n$ on defined above.
Here, $\Omega=\{e_1,\dots,e_n\}$, so the constraints tells us $\sum_{k=1}^n p_k e_k = \theta$, so the minimization is over the single point~$p=\theta$.

This setup has an application to what one might call `generalized permanents.'

\begin{exa}[Generalized permanents and van der Waerden's theorem]
Consider the following quantity, introduced by Barvinok~\cite{Ba10} and studied by Gurvits~\cite{Gu15}.
For vectors $r \in \QQ_{\geq 0}^n$, $c \in \QQ_{\geq 0}^m$, let $\I(r,c)\subset \Mat_{n\times m}(\ZZ_{\geq0})$ denote the set of nonnegative integer matrices with row sums equal to~$r$ and column sums equal to~$c$ (which can be empty).
For a matrix $M \in \Mat_{n\times m}(\RR_{\geq0})$ let
\begin{align*}
  \perm_{r, c}(M) :=  \sum_{B \in \I(r, c)} \prod_{i \in [n], j \in [m]} M^{B_{ij}}/B_{ij}!
\end{align*}
where the empty sum is taken to be zero.
In particular, for $n = m$ and denoting by~$\mathbf 1$ the all-ones vector, $\perm_{\mathbf{1}, \mathbf{1}}(M)$ is the \emph{permanent} $\perm(M)$ of $M$.
One has the following formula for the exponential decay of $\perm_{t\mathbf{1}, t\mathbf{1}}(M)$ in~$t$:
\begin{align}\label{eq:perm dual}
  \limsup_{k \to \infty}\left(k!\perm_{(k/n)\mathbf{1}, (k/n)\mathbf{1}}(M)\right)^{1/k}
= \inf_{x,y \in \RR_{> 0}^n}  \frac{\sum_{ij\in [n]} M_{ij} x_iy_j}{\left(\prod_{i\in [n]} x_i y_i\right)^{1/n}}.
\end{align}
One proves this by applying \cref{thm:cap theta main} to the action of $G = \CC^n_\times \times \CC^n_\times$ on $V = \Mat_n(\CC)$ by pre- and post-multiplication.
Then, $\perm_{t\mathbf{1}, t\mathbf{1}} (M) = \norm{\Pi_{tn, (t\mathbf{1}, t\mathbf{1})}(\sqrt{M} ^{\ot tn})}^2$, where~$\sqrt{M}$ denotes the entrywise square root of~$M$, while the right-hand side of \cref{eq:perm dual} is simply~$\capa^2_{(\mathbf{1}/n, \mathbf{1}/n)}(\sqrt{M})$.
In particular,
\begin{align*}
  \perm(M) \leq \frac1{n!} \capa^{2n}_{(\mathbf{1}/n, \mathbf{1}/n)}(\sqrt{M}).
\end{align*}
The Van der Waerden theorem for the permanent of a doubly stochastic matrix supplies the reverse inequality $\perm(M) \geq \capa^{2n}_{(\mathbf{1}/n, \mathbf{1}/n)}(\sqrt{M}) n!/n^{2n}$~\cite{egoryvcev1980cyr}.
More generally, for $\sum_{i=1}^n r_i = \sum_{j=1}^m c_j = 1$ and $M$ a nonnegative $n\times m$ matrix,
\begin{align*}
  \limsup_{k \to \infty}(k!\perm_{kr, kc}(M))^{\frac1k}
= \inf_{x \in \RR_{> 0}^n, y \in \RR_{> 0}^m} \frac{\sum_{i\in [n], j \in [m]} M_{ij} x_i y_j}{\left( \prod_{i\in [n]} x_i^{r_i} \right)\left( \prod_{j\in [m]} y_j^{c_j} \right)}.
\end{align*}
For a discussion of reverse inequalities for other $k, r, c$, see~\cite{Gu15}.
\end{exa}

\subsubsection*{Keyl and Werner's results on quantum tomography.}
Let $G = \GL(n)$, $K=\U(n)$ act on square complex matrices $A \in \Mat_n(\CC)$ by left multiplication; in this case $\mu(A) = A A^\dagger / \norm{A}_F^2$.
Keyl and Werner established large-deviations principles for~$X_k(A)$ and~$Y_k(A)$, motivated by the problem of estimating an unknown quantum state and its eigenvalues, respectively~\cite{keyl2006quantum,keyl2005estimating} (for optimality and variations see~\cite{odonnell2016efficient,haah2017sample}).

We may assume $\norm{A}_F=1$.
Note that the random variables depend only on~$\sigma = AA^\dagger$; accordingly we write $X_k(\sigma)$ and $Y_k(\sigma)$.
We may identity $(i\ka)^*$ with the Hermitian~$n\times n$ matrices, and choose $C_+$ as the real diagonal matrices with decreasing diagonal, which we identify with a subset of~$\RR^n$; then the map~$s\colon(i\ka)^*\to C_+$ sends a Hermitian matrix to its sorted eigenvalues.
Accordingly, $B^-$ is the invertible lower triangular matrices.
Finally, let $\mathcal D_n$ denote the set of positive semidefinite $n\times n$ matrices with unit trace.
Since $\Delta(\sigma) \subseteq \mathcal D_n$, it suffices to compute the rate function for this set.

\begin{thm}[Keyl]
The $\mathcal D_n$-valued random variables $X_k(\sigma)$ satisfy a large deviations principle with rate function~$I(\cdot\Vert\sigma)$, defined for all $\rho = u\diag(p)u^\dagger$ with $u\in U(n)$ and $p\in\mathcal P_n \cap C_+$ by
\begin{align*}
  I(\rho\Vert\sigma) = \sum_{k=1}^n \left( p_k \log p_k - (p_k - p_{k+1}) \log \prim_k(u^\dagger \sigma u) \right),
\end{align*}
where $\prim_k$ denotes the $k^{th}$ principal minor (upper left $k\times k$ subdeterminant), $0\log 0:=0$, $p_{n+1}:= 0$.
\end{thm}
\begin{proof}
It suffices to show that $-\log\capa_\rho(A) = I(\rho\Vert\sigma)$ for $AA^\dagger = \sigma$.
Then the infimum in $\capa_\rho(A) = \inf_{b\in B^-} \, \abs{\chi_{-p}(b)} \, \norm{bu^\dagger A}_F$ is attained for any lower triangular matrix such that $\mu([bu^\dagger A]) = p$, and if there is no such matrix then $\capa_\rho(A) = 0$.
If~$\rho$ is positive definite, this follows from the discussion in \cref{subsec:moment map}; if not one can restrict to the support of $\rho$ and reason analogously there (as in~\cite{burgisser2018efficient}).
We can choose $b := \diag(\sqrt p) c^+$, where $u^\dagger \sigma u = cc^\dagger$ with $c$ is a Cholesky decomposition and $c^+$ denotes a matrix that is the inverse of $c$ on the support of $cc^\dagger$. Then it also holds that~$\norm{bu^\dagger A}_F=1$, so we obtain
\begin{align*}
  -\log\capa^2_\rho(A)
&= -\log \, \abs{\chi_{-p}(b)}^2
= \sum_{k=1}^n p_k \log \, \abs{b_{kk}}^2
= \sum_{k=1}^n (p_k - p_{k+1}) \log \, \abs{\prim_k(b)}^2 \\
&= \sum_{k=1}^n p_k \log p_k - \sum_{k=1}^n (p_k - p_{k+1}) \log \, \abs{\prim_k(c)}^2
= I(\rho\Vert\sigma),
\end{align*}
where the last step follows since $\abs{\prim_k(c)}^2 = \prim_k(u^\dagger \sigma u)$ for all~$k\in[n]$.
\end{proof}

We remark that $I(\rho\Vert\sigma)$ is upper bounded by the quantum relative entropy~$D(\rho\Vert\sigma) = \tr(\rho\log\rho) -\tr(\rho\log\sigma)$~\cite{keyl2006quantum}.
The following corollary computes the rate for the random variables~$Y_k(\sigma)$. We note that $s(\sigma)$ is necessarily a probability distribution since $\sigma$ is positive semidefinite with unit trace.

\begin{cor}[Keyl-Werner]
The random variables~$Y_k(\sigma)$ satisfy a large deviations principle with rate function $\DKL(\cdot\Vert q)$, where $q \in \mathcal P_n \cap C_+$ is the vector of eigenvalues of $\sigma$, sorted decreasingly.
\end{cor}
\begin{proof}
We need to show that
\begin{align*}
  \inf_{u\in\U(n)} I(u\diag(p)u^\dagger\Vert\sigma) = \DKL(p\Vert q).
\end{align*}
By the Cauchy interlacing theorem, the eigenvalues $\mu_1,\dots,\mu_k$ of the $k\times k$ principal submatrix of~$u^\dagger\sigma u$ satisfy $\mu_i \leq q_i$ for $i\in[k]$.
Thus, $\prim_k(u^\dagger\sigma u) \leq \prod_{i=1}^k q_i$, and equality is achieved if we choose~$u\in\U(n)$ such that $u^\dagger\sigma u = \diag(q)$.
\end{proof}

\subsubsection*{Duffield's theorem on tensor power multiplicities}
Returning to the general case, consider an arbitrary representation of~$G$ on a vector space~$W$, and denote by~$n_{k,\lambda}$ the multiplicity of~$V_\lambda$ in the tensor power representation~$W^{\ot k}$.
Then we can consider the following sequence of probability measures,
\begin{align*}
  \nu_{W,k} = \frac1{\sum_{\lambda\in\Lambda_+} n_{k,\lambda}} \sum_{\lambda\in\Lambda_+} n_{k,\lambda} \, \delta_{\frac\lambda k},
\end{align*}
which are defined similarly to \cref{eq:nu heckman} but with~$W^{\ot k}$ in place of the symmetric subspace.
Duffield proved the following large deviations principle for irreducible~$W$~\cite{duffield1990large} (cf.~\cite{postnova2020multiplicities}).

\begin{thm}[Duffield]\label{thm:duffield}
The probability measures~$\nu_{W,k}$ on~$C_+$ satisfy a large deviations principle with rate function
\begin{align*}
  I(\theta) = \sup_{H\in i\ta_+} \left( \theta(H) - \log \frac{\chi_W(e^H)}{d_W} \right),
\end{align*}
where $\chi_W$ denotes the character of~$W$, $d_W$ its dimension, and $i\ta_+ \subseteq i\ta$ the Weyl chamber dual to $C_+$.
\end{thm}

To relate the setup of Duffield with our setting, consider the representation~$V=\End(W)$, equipped with the induced left action of~$G$ and the Hilbert-Schmidt inner product, and let $v = I_W / \sqrt{d_W} \in V$ be the normalized identity operator.
(For $G=\GL(n)$ acting on $W=\CC^n$, we recognize this as a special case of the Keyl-Werner setup.)
The measure~$\nu_{W,k}$ is not the same as the distribution of the random variable~$Y_k(v)$, since the latter takes value $\lambda/k$ with probability~$\norm{\Pi_{k,\lambda} v^{\ot n}}^2 = d_\lambda n_{k,\lambda} / d_W^n$.
Nevertheless, both sequences of probability measures have the same asymptotic behavior and large deviations rate.
This follows from the Weyl dimension formula, which implies that $d_\lambda \leq \poly(k)$ for all~$V_\lambda$ that appear in~$V^{\ot k}$.
To prove \cref{thm:duffield}, it thus suffices to verify that $-\log\Capa^2_\theta(v) = I(\theta)$.
Indeed,
\begin{align*}
  -\log\Capa^2_\theta(v)
 = -\log\capa^2_\theta(v)
 &= \sup_{H\in i\ta} \left( \theta(H) - \log \inf_{n \in N^-} \frac{\tr(e^{\Phi(H)} \phi(n)\phi(n)^\dagger)}{d_W} \right) \\
 &= \sup_{H\in i\ta} \left( \theta(H) - \log \frac{\tr e^{\Phi(H)}}{d_W} \right) = I(\theta).
\end{align*}
Here we first used that $\capa_\theta(v)$ is $K$-invariant for $v=I_W/\sqrt{d_W}$, then the infimum over $n\in N^-$ is obtained at the identity element as~$\phi(n)$ is unipotent, and the last step holds because the objective is invariant under the Weyl group.

\subsection{Mathieu conjecture}\label{subsec:mathieu}

We now discuss the relation between our results and the Mathieu conjecture in more detail.
Recall that a function $f\colon K \to \CC$ is called \emph{$K$-finite} if it takes the form $f(u) = \braket{v, \phi(u) w}$ for some finite-dimensional representation $\phi\colon G\to\GL(V)$ and vectors $v, w\in V$.
If we may take $v = w$ then $f$ is said to be \emph{positive-definite}.
We define the \emph{constant term} of~$f$ by
\begin{align*}
  \cst(f) = \int_K f(u) \, \du.
\end{align*}
Note that $\cst(f)$ is the constant term of the Fourier series of $f$ in the sense of the Peter-Weyl theorem.
The following conjecture, due to Mathieu~\cite{mathieu1995some}, is known to imply Keller's Jacobian conjecture.

\begin{conj}[Mathieu] Let $K$ be a compact connected Lie group, and let $f$ and $g$ be $K$-finite functions on $K$. If $\cst(f^k) = 0$ for all $k$, then $\cst(f^k g)=0$ for all but finitely many $k$.
\end{conj}

Mathieu's conjecture has been proved in the Abelian case by Duistermaat and van der Kallen~\cite{duistermaat1998constant}. Their \cref{eq:vdk} may be seen as a refinement of the conjecture in the Abelian case.

We now explain how to similarly treat the positive-definite case.
For a positive-definite function $f(u) = \braket{v, \phi(u) v}$, the constant terms of its powers are directly related to the projection to the trivial component, as follows from \cref{eq:pi_k as integral}:
\begin{align}\label{eq:cst vs norm}
  \cst(f^k) = \int_K \braket{v, \phi(u) v}^k \, \du = \norm{\Pi_k v^{\ot k}}^2.
\end{align}
We now prove the Mathieu conjecture in the positive-definite case by using the moment polytope.

\begin{prp}
Let $f,g$ be $K$-finite functions and suppose $f$ is positive-definite. If $\cst(f^k) = 0$ for all $k$, then $\cst(f^k g)=0$ for all but finitely many $k$.
\end{prp}
\begin{proof}
Suppose $f(u) = \braket{v, \phi(u) v}$ as above and $\cst(f^k) = 0$ for all $k\in\NN$.
By \cref{eq:pi_k as integral}, every homogeneous invariant polynomial vanishes on~$v$ and hence $0$ is not in the moment polytope~$\Delta(v)$ of~$v$, as mentioned in~\cref{subsec:moment map}.
Thus there exists $\eps>0$ such that $\Pi_{k,\lambda} v^{\ot k} = 0$ unless $\norm{\lambda} \geq k \eps$.

Consider the Fourier transform $\hat{f}= \bigoplus_{\lambda} \widehat{f}(\lambda) \in \bigoplus_{\lambda} \End(V_\lambda)$ of $f$, where $\lambda$ runs over $\Lambda_+$ and $\hat{f}(\lambda):= \int_{K} \phi_\lambda^\dagger(u) f(u) \du$.
One verifies that $d_\lambda \norm{\widehat{f^k}(\lambda)}_{\text{HS}}^2 = \norm{\Pi_{k,\lambda} v^{\ot k}}^2$.
Therefore, $\widehat{f^k}(\lambda) = 0$ unless~$\norm{\lambda} \geq k \eps$.
On the other hand, since $g$ is $K$-finite, its Fourier transform has finite support, and the same is true for~$\bar g$.
Thus we have an upper bound~$N$ such that $\widehat{\bar g}(\lambda) = 0$ for all $\norm{\lambda}\geq N$.
By Parseval's formula, it follows that
\begin{align*}
  \cst(f^k g) = \int f^k(u) g(u) \, \du = \braket{\bar g, f^k}_{L^2(K)} = \sum_\lambda d_\lambda \tr\bigl[ \hat{\bar g}(\lambda)^\dagger \hat f(\lambda) \bigr] = 0
\end{align*}
as soon as $k\eps > N$.
\end{proof}

While the above proof shows that Mathieu's conjecture holds for positive-definite $K$-finite functions, it does not immediately imply a formula analogous to \cref{eq:vdk} for $\limsup_{k \to \infty}\cst(f^k)^{1/k}$ where $f$ is a positive-definite $K$-finite function.
Suitably reinterpreted, \cref{thm:main} provides such a formula completely.
Note that $f$ may be treated as a function on~$G$.

\begin{thm}[\cref{thm:main} refomulated]\label{thm:main mathieu}
Let $f$ be a positive-definite $K$-finite function.
Then:
\begin{align*}
  \limsup_{k \to \infty}\cst(f^k)^{1/k} = \inf_{p \in P} f(p),
\end{align*}
where $P = \exp(i\ka) \subseteq G$.
\end{thm}

\noindent Here $G$, the complexification of $K$, plays the role of $\CC_\times$, and $\inf_{p\in P} f(p) > 0$ plays the role of the critical value $\nu$ in \cref{eq:vdk}.
Generalizing \cref{thm:main mathieu} to all $K$-finite functions in a way that sheds light on the Mathieu conjecture remains a tantalizing open problem.


\section*{Acknowledgments}
We thank Michel Brion, Matthias Christandl, Ankit Garg, Shrawan Kumar, Rafael Oliveira, Paul-Emile Paradan, and Mich\`ele Vergne for interesting conversations.
We are grateful to Geordie Williamson and Jean-Beno\^it Bost for pointing out Ref.~\cite{zhang1994geometric}.

MW acknowledges support by the NWO through Veni grant~680-47-459 and OCENW.KLEIN.267, by the Deutsche Forschungsgemeinschaft (DFG, German Research Foundation) under Germany's Excellence Strategy - EXC\ 2092\ CASA - 390781972, by the BMBF through project Quantum Methods and Benchmarks for Resource Allocation (QuBRA), and by the European Research Council~(ERC) through ERC Starting Grant 101040907-SYMOPTIC.

\bibliographystyle{amsplain}
\bibliography{capdual}

\end{document}